\documentclass[a4paper,portuguese,twoside,thmsa,12pt]{article}
\usepackage{amsmath}
\usepackage{amsfonts}
\usepackage[latin1]{inputenc}
\usepackage{amssymb}
\usepackage{latexsym}
\usepackage{makeidx}
\usepackage{indentfirst}
\usepackage{amscd}
\usepackage{eufrak}
\usepackage[latin1]{inputenc}
\usepackage{amsmath}

\DeclareMathOperator{\supp}{supp}

\newtheorem{theorem}{Theorem}

\newtheorem{corollary}[theorem]{Corollary}

\newtheorem{lemma}[theorem]{Lemma}

\newenvironment{proof}[1][Proof]{\noindent\textbf{#1:} }{\hfill \rule{0.5em}{0.5em}}

\begin{document}

\title{A Generic Property of Exact Magnetic Lagrangians}
\author{Mário Jorge Dias Carneiro and Alexandre Rocha.}
\maketitle

\begin{abstract}
We prove that for the set of Exact Magnetic Lagrangians the property
\textquotedblleft There exist finitely many static classes for every
cohomology class" is generic. We also prove some dynamical consequences of this property.
\end{abstract}

\section{Introduction}

Let $M$ be a closed manifold equipped with an Riemannian metric $%
g=\left\langle .,.\right\rangle $. A Lagrangian $L:TM\rightarrow \mathbb{R}$
is called Exact Magnetic Lagrangian if
\begin{equation*}
L\left( x,v\right) =\frac{\left\Vert v\right\Vert ^{2}}{2}+\left\langle \eta
,v\right\rangle
\end{equation*}
for some non-closed 1-form $\eta $.

This type of Lagrangian fits into Mather's theory, as developed by R. Mañé
and A. Fathi, about Tonelli Lagrangians, namely, it is fiberwise convex and
superlinear. We refer the reader to the references Fathi in \cite{fa1},
Contreras and Iturriaga in \cite{gon2} for expositions of this theory.

Let $\mathfrak{M}\left( L\right) $ be the set of action minimizing measures.
Recall that $\mathfrak{M}\left( L\right) $ is the set of $\mu $ Borel
probability measures in $TM$ which are invariant under the Euler-Lagrange
flow $\varphi _{t}$ generated by $L$ and minimizes the \textit{action}, that
is for all invariant probability $\nu $ in $TM$:
\begin{equation*}
\int_{TM}Ld\mu \leq \int_{TM}Ld\nu .
\end{equation*}

The set $\mathfrak{M}\left( L\right) $ is a simplex whose extremal points
are the ergodic minimizing measures.


Since the Euler Lagrange flow generated by $L$ does not change by adding a
closed one form $\zeta $, we also consider the action minimizing measures $%
\mathfrak{M}\left( L-\zeta \right) $. The minimal action value, depends only
on the cohomology class $c=[\zeta ]\in H^{1}(M,{\mathbb{R}})$ of the closed
one form, so it is denoted by $-\alpha (c)$. It is known that $\alpha (c)$
is the energy level that contains the \textit{Mather set for the cohomology
class }$c$:%
\begin{equation*}
{\widetilde{\mathcal{M}}}_{c}=\bigcup_{\mu \in \mathfrak{M}(L-\zeta )}\supp%
(\mu ).
\end{equation*}%
${\widetilde{\mathcal{M}}}_{c}$ is a compact invariant set which is a graph
over a compact subset $\mathcal{M}_{c}$ of $M$, the projected Mather set
(see \cite{mat1}). $\mathcal{M}_{c}$ is laminated by curves, which are
global (or time independent) minimizers. Mather also proved that the
function $c\mapsto \alpha (c)$ is convex and superlinear.

In general, ${\widetilde{\mathcal{M}}}_{c}$ is contained in another compact
invariant set, which also a graph whose projection is laminated by global
minimizers: the \textit{Aubry set for the cohomology class c}, denoted by $%
\widetilde{\mathcal{A}}_{c}$. Mañé proved that $\widetilde{\mathcal{A}}_{c}$
is chain recurrent and it is a challenging question to describe the dynamics
of the Euler-Lagrange flow restricted to $\widetilde{\mathcal{A}}_{c}$. The
definition of Aubry set and some its properties are given in Section \ref%
{sec3}.

Of course this question only makes sense if it is posed for generic
Lagrangians, since many pathological examples can be constructed. The notion
of genericity in the context of Lagrangian systems is provided by Mañé in
\cite{man2}. The idea is to make special perturbations by adding a
potential: $L(x,v)+\Psi (x)$, for $\Psi \in C^{\infty }(M)$.

A property is \textit{generic} in the sense of Mañé if it is valid for all
Lagrangians $L(x,v)+ \Phi(x)$ with $\Phi$ contained in a residual subset $%
\mathcal{O}$.

In this setting, G. Contreras and P. Bernard proved in the work \textit{A
Generic Property of Families of Lagrangian Systems} (see \cite{ber1}) that
generically, in the sense of Mañé, for all cohomology class $c$ there is
only a finite number of minimizing measures. This theorem is a consequence
of an abstract result which is useful in different situations.

In general, when we are dealing with an specific class of Lagrangians,
perturbations by adding a potential are not allowed. However, due to the
abstract nature of Bernard-Contreras proof it may be addapted to the
specific case like the one treated here.

The objective of this paper is to prove the genericity of finitely many
minimizing measures for Exact Magnetic Lagrangians and apply it to the
dynamics of the Aubry set.

Let us consider $\Gamma ^{1}\left( M\right) $ the set of smooth 1-forms in $%
M\ $endowed with the metric%
\begin{equation}
d\left( \omega _{1},\omega _{2}\right) =\sum_{k\in \mathbb{N}}\frac{\arctan
\left( \left\Vert \omega _{1}-\omega _{2}\right\Vert _{k}\right) }{2^{k}},
\label{d1}
\end{equation}%
denoting by $\left\Vert \omega \right\Vert _{k}$ the $C^{k}$-norm of the
1-form $\omega .$ With this metric $\Gamma ^{1}\left( M\right) $ is a
Frechet space, it means that $\Gamma ^{1}\left( M\right) $ is a locally
convex topological vector space whose topology is defined by a
translation-invariant metric, and that $\Gamma ^{1}\left( M\right) $ is
complete for this metric. 

The main result of this paper is the following:

\begin{theorem}
\label{teo1}Let $A$ be a finite dimensional convex family of Exact Magnetic
Lagrangians. Then there exists a residual subset $\mathcal{O}$ of $\Gamma
^{1}\left( M\right) $ such that,%
\begin{equation*}
\omega \in \mathcal{O},L\in A\Rightarrow \dim \mathfrak{M}\left( L+\omega
\right) \leq \dim A.
\end{equation*}%
Hence there exist at most $1+\dim A$ ergodic minimizing measures of $%
L+\omega .$
\end{theorem}

\begin{corollary}
Let $L$ be a Exact Magnetic Lagrangian. Then there exists a residual subset $%
\mathcal{O}$ of $\Gamma ^{1}\left( M\right) $ such that for all $c\in
H^{1}\left( M,\mathbb{R}\right) \ $and for all $\omega \in \mathcal{O},$
there are at most $1+\dim H^{1}\left( M,\mathbb{R}\right) $ ergodic
minimizing measures of $L+\omega -c.$
\end{corollary}

The last part of this work is dedicated to prove some consequences about the
dynamics. For instance, using the work of Contreras and Paternain, \cite%
{gon1} we obtain connecting orbits between the elements of the Aubry set
that contain the support of minimizing measures (the so called
\textquotedblleft static classes").


%
%
%
%
%
%


\section{Adapting the abstract setting of Bernard and Contreras}

As it was pointed out previously, the proof of Theorem \ref{teo1} is an
application of the work of Contreras and Bernard. Here we state their result.

Assume that we are given

\begin{description}
\item[(i)] Three topological vector spaces $E,F,G.$

\item[(ii)] A continuous linear map $\pi :F\rightarrow G.$

\item[(iii)] A bilinear map $\left\langle ,\right\rangle :E\times
G\rightarrow
\mathbb{R}
$.

\item[(iv)] Two metrizable convex compact subsets $H\subset F$ and $K\subset
G$ such that $\pi \left( H\right) \subset K.$
\end{description}

Suppose that

\begin{enumerate}
\item The restriction of the map given by (iii), $\left\langle
,\right\rangle |_{E\times K}$ is continuous.

\item The compact $K$ is separated by $E.$ This means that, if $\mu $ and $%
\nu $ are two different points of $K,$ then there exists a point $\omega $
in $E$ such that $\left\langle \omega ,\mu -\nu \right\rangle \neq 0.$

\item $E$ is a Frechet space.
\end{enumerate}

Note then that $E$ has the Baire property, that is any residual subset of $E$
is dense.

We shall denote by $H^{\ast }$ the set of affine and continuous functions
defined on $H.$ Given $\bar{L}\in H^{\ast }$ denote by%
\begin{equation*}
M_{H}\left( \bar{L}\right) =\arg \min \bar{L}
\end{equation*}%
the set of points $\alpha \in H$ which minimizes $\bar{L}|_{H},$ and by $%
M_{K}\left( \bar{L}\right) $ the image $\pi \left( M_{H}\left( \bar{L}%
\right) \right) .$ These are compacts convex subsets of $H$ and $K.$

Under these conditions we have:

\begin{theorem}[G. Contreras and P. Bernard]
\label{teo2}For every finite dimensional affine subspace $A$ of $H^{\ast }$,
there exists a residual subset $\mathcal{O}\left( A\right) \subset E$ such
that, for all $\omega \in \mathcal{O}\left( A\right) $ and $\bar{L}\in A,$
we have%
\begin{equation*}
\dim M_{K}\left( \bar{L}+\omega \right) \leq \dim A
\end{equation*}
\end{theorem}

In order to apply this theorem, we need to define the above objects in an
adequate setting as follows:

Let $C$ be the set of continuous functions $f:TM\rightarrow
\mathbb{R}
$ with linear growth, that is%
\begin{equation}
\left\Vert f\right\Vert _{\ell in}=\sup_{\theta \in TM}\frac{\left\vert
f\left( \theta \right) \right\vert }{1+\left\vert \theta \right\vert }%
<+\infty  \label{d2}
\end{equation}%
endowed with the norm $\left\Vert .\right\Vert _{\ell in}.$

We define:

\begin{list}{\textbf{$\bullet$}}{\usecounter{quest}
\setlength{\labelwidth}{-2mm} \setlength{\parsep}{0mm}
\setlength{\topsep}{0mm} \setlength{\leftmargin}{0mm}}
\renewcommand{\labelenumi}{(\alph{enumi})}

\item $E=\Gamma ^{1}\left( M\right) $ endowed with the metric $d$ defined in
(\ref{d1}).

\item $F=C^{\ast }$ is the vector space of continuous linear functionals $%
\mu :C\rightarrow
\mathbb{R}
$ provided with the weak-$\star $ topology:%
\begin{equation*}
\lim_{n}\mu _{n}=\mu \Leftrightarrow \lim_{n}\mu _{n}\left( f\right) =\mu
\left( f\right) ,\forall f\in C.
\end{equation*}

\item $G$ is the vector space of continuous linear functionals $\mu :\Gamma
^{0}\left( M\right) \rightarrow
\mathbb{R}
,$ where $\Gamma ^{0}\left( M\right) $ is the space of continuous 1-forms on
$M$. 
%
Note that the Riemannian metric $g=\left\langle .,.\right\rangle $ allows us
to represent any continuous 1-form as $\left\langle X,.\right\rangle ,$ for
some $C^{0}$ vector field $X.$ We endow $G$ with the weak-$\star $ topology:%
\begin{equation*}
\lim_{n}\mu _{n}=\mu \Leftrightarrow \lim_{n}\mu _{n}\left( \omega \right)
=\mu \left( \omega \right) ,\forall \omega \in \Gamma ^{0}\left( M\right) .
\end{equation*}

\item The continuous linear $\pi :F\rightarrow G$ is given by%
\begin{equation*}
\pi \left( \mu \right) =\mu |_{\Gamma ^{0}\left( M\right) }.
\end{equation*}

\item For a given natural number $N$, let
\begin{equation*}
B_{N}=\left\{ \left( x,v\right) \in TM:\left\vert v\right\vert \leq
N\right\} .
\end{equation*}

Let us denote by $M_{N}^{1}$ the set of the probability measures $\mu $ in $%
TM$ such that $\supp\mu \subset B_{N}.$ Define $K_{N}=\pi \left(
M_{N}^{1}\right) \subset G$, the restriction of the probabilities in $%
M_{N}^{1}$ to $\Gamma ^{0}\left( M\right) $.
\end{list}

\noindent \textbf{Claim 1. }$K_{N}$ is metrizable.

\begin{proof}
Since $G$ is the dual of $\Gamma ^{0}\left( M\right) $, we define a norm in $%
G$ as follows%
\begin{equation*}
\left\Vert \mu \right\Vert _{G}=\sup_{\left\Vert \omega \right\Vert _{\ell
in}\leq 1}\left\{ \left\vert \mu \left( \omega \right) \right\vert \right\} .
\end{equation*}%
If $\mu \in K_{N},$%
\begin{eqnarray*}
\left\Vert \mu \right\Vert _{\emph{G}} &=&\sup_{\left\Vert \omega
\right\Vert _{\ell in}\leq 1}\left\{ \left\vert \int_{TM}\omega d\mu
\right\vert \right\} \leq \sup_{\left\Vert \omega \right\Vert _{\ell in}\leq
1}\left\{ \int_{TM\cap B_{N}}\left\vert \omega \right\vert d\mu \right\} \\
&=&\sup_{\left\Vert \omega \right\Vert _{\ell in}\leq 1}\left\{ \int_{B_{N}}%
\frac{\left\vert \omega \left( x,v\right) \right\vert }{1+N}\left(
1+N\right) d\mu \right\} \leq \sup_{\left\Vert \omega \right\Vert _{\ell
in}\leq 1}\left\{ \int_{B_{N}}\frac{\left\vert \omega \left( x,v\right)
\right\vert }{1+\left\vert v\right\vert }\left( 1+N\right) d\mu \right\} \\
&\leq &\left( N+1\right) \sup_{\left\Vert \omega \right\Vert _{\ell in}\leq
1}\left\{ \int_{TM}\left\Vert \omega \right\Vert _{\ell in}d\mu \right\}
\leq N+1.
\end{eqnarray*}%
This shows that $K_{N}\subset B_{G},$ where $B_{G}$ is the ball of radius $%
N+1$ in $G=\Gamma ^{0}\left( M\right) ^{\ast }.$ Then, by following
classical theorem of Analysis, it is enough show that $\Gamma ^{0}\left(
M\right) $ is a separable vector space.

\begin{theorem}
\label{teoa1}Let $E$ a Banach's space. Then $E$ is separable if, and only
if, the unit ball $B_{E^{\ast }}\subset E^{\ast }$ in the weak-$\star $
topology is metrizable.
\end{theorem}

The separability of $\Gamma ^{0}\left( M\right) $ follows from the lemma
below and of the duality between 1-forms and vector fields provided by the
Riemannian metric.

\begin{lemma}
\label{lemmaa1}The space $\mathfrak{X}^{0}\left( M\right) $ of continuous
vector fields in a compact manifold $M$ is separable.
\end{lemma}

\begin{proof}
By compactness of $M,$ we can consider a number finite local trivializations
$\hat{U}_{i}\subset TM\rightarrow U_{i}\times
\mathbb{R}
^{n}$ of the tangent bundle $TM$ and by compactness of $\overline{U_{i}},$ $%
\mathfrak{X}^{0}\left( \overline{U_{i}}\right) =C^{0}\left( \overline{U_{i}},%
\mathbb{R}
^{n}\right) $ is separable. Let $\left\{ f_{n}^{i}\right\} $ be a dense
subset in $\mathfrak{X}^{0}\left( \overline{U_{i}}\right) $ and $\left\{
\alpha _{i}\right\} $ a partition of unity subordinate to the open cover $%
\left\{ U_{i}\right\} .$ It is enough show that $\left\{ \sum_{i}\alpha
_{i}f_{n}^{i}\right\} $ is dense in $\mathfrak{X}^{0}\left( M\right) .$ Let $%
g\in \mathfrak{X}^{0}\left( M\right) $ and consider $g_{i}=\alpha _{i}g$.
Then $g=\sum \alpha _{i}g=\sum g_{i},$ $\supp g_{i}\subset U_{i}\subset
\overline{U_{i}}.$ Given $\epsilon >0$ there exists $n_{i}\in
\mathbb{N}
$ such that%
\begin{equation*}
\left\Vert f_{n_{i}}^{i}-g_{i}\right\Vert <\frac{\epsilon }{2^{i}}.
\end{equation*}%
Then%
\begin{eqnarray*}
\left\Vert \sum_{i}\alpha _{i}f_{n_{i}}^{i}-g\right\Vert &=&\left\Vert
\sum_{i}\alpha _{i}f_{n_{i}}^{i}-\sum_{i}\alpha _{i}g\right\Vert \leq
\sum_{i}\sup_{\overline{U_{i}}}\left\vert f_{n_{i}}^{i}-g_{i}\right\vert \\
&=&\sum_{i}\left\Vert f_{n_{i}}^{i}-g_{i}\right\Vert <\sum_{i}\frac{\epsilon
}{2^{i}}<\epsilon .
\end{eqnarray*}
\end{proof}

Let us consider $\left( X_{n}\right) $ a dense sequence in $\mathfrak{X}^{0}$%
$\left( M\right) $ and $\omega _{n}=\left\langle X_{n},\cdot \right\rangle
\in \Gamma ^{0}\left( M\right) .$ Let $\omega =\left\langle X,\cdot
\right\rangle \in \Gamma ^{0}\left( M\right) $ and $\mathcal{U}_{\varepsilon
}$ be a ball in $\Gamma ^{0}\left( M\right) $ centered at $\omega .$, of
radius $\varepsilon >0$. Then there exists a $X_{n}\in V_{\varepsilon
}\left( X\right) ,$ where $V_{\varepsilon }\left( X\right) $ is the ball in $%
\mathfrak{X}^{0}\left( M\right) $ of radius $\varepsilon $ and center $X.$
It follows that%
\begin{eqnarray*}
\left\Vert \omega _{n}-\omega \right\Vert _{\ell in} &=&\sup_{\left(
x,v\right) \in TM}\frac{\left\vert \left( \omega _{n}-\omega \right) \left(
x,v\right) \right\vert }{1+\left\vert v\right\vert }=\sup_{\left( x,v\right)
\in TM}\frac{\left\vert \left\langle \left( X_{n}-X\right) \left( x\right)
,v\right\rangle \right\vert }{1+\left\vert v\right\vert } \\
&\leq &\sup_{\left( x,v\right) \in TM}\frac{\left\vert \left( X_{n}-X\right)
\left( x\right) \right\vert \left\vert v\right\vert }{1+\left\vert
v\right\vert }\leq \sup_{x\in M}\left\vert \left( X_{n}-X\right) \left(
x\right) \right\vert <\varepsilon .
\end{eqnarray*}%
This shows that $\omega _{n}\in $ $\mathcal{U}_{\varepsilon }$ and $\Gamma
^{0}\left( M\right) $ is separable, so $K_{N}$ is metrizable. This finishes
the proof of the Claim 1.
\end{proof}

Observe that $K_{N}$ is compact and convex since $K_{N}=\pi \left(
M_{N}^{1}\right) ,$ $\pi $ is a continous map and $M_{N}^{1}$ is a compact
subset of probability measures in $TM.$

\begin{list}{\textbf{$\bullet$}}{\usecounter{quest}
\setlength{\labelwidth}{-2mm} \setlength{\parsep}{0mm}
\setlength{\topsep}{0mm} \setlength{\leftmargin}{0mm}}
\renewcommand{\labelenumi}{(\alph{enumi})}

\item \noindent \noindent The bilinear mapping $\left\langle ,\right\rangle
:E\times G\rightarrow
\mathbb{R}
$ is given by integration:%
\begin{equation*}
\left\langle \omega ,\mu \right\rangle =\int_{TM}\omega d\mu .
\end{equation*}%
Note that here we apply the Hahn-Banach Theorem for extends the functional $%
\mu $ and that the above integral does not depend on the extension of $\mu $
to a signed measure on $TM$ given by Riesz representation Theorem. Moreover,
\begin{equation*}
\left\langle ,\right\rangle :E\times K_{N}\rightarrow
\mathbb{R}
\end{equation*}%
is continuous. In fact, if $\omega _{n}\rightarrow \omega $ and $\mu
_{n}\rightarrow \mu $ with $\left( \omega _{n}\right) \subset E$ and $\left(
\mu _{n}\right) \subset K_{N},$ then%
\begin{equation*}
\lim_{n}\int_{TM}\eta d\mu _{n}=\int_{TM}\eta d\mu ,\forall \eta \in E,
\end{equation*}%
and $d\left( \omega _{n},\omega \right) \rightarrow 0$ implies that given $%
\epsilon >0,$ there exists $n_{0}\in
\mathbb{N}
$ such that%
\begin{equation*}
\forall n\geq n_{0},\left\Vert \omega _{n}-\omega \right\Vert _{\ell in}<%
\frac{\epsilon }{\left( N+1\right) }.
\end{equation*}%
Since $\mu _{n},\mu \in K_{N},$ we have%
\begin{eqnarray*}
\left\vert \int_{TM}\omega _{n}d\mu _{n}-\int_{TM}\omega d\mu
_{n}\right\vert  &\leq &\int_{B_{N}}\left\vert \omega _{n}-\omega
\right\vert d\mu _{n} \\
&=&\int_{B_{N}}\frac{\left\vert \omega _{n}-\omega \right\vert }{1+N}\left(
1+N\right) d\mu _{n} \\
&\leq &\left( 1+N\right) \int_{B_{N}}\frac{\left\vert \omega _{n}-\omega
\right\vert }{1+\left\vert v\right\vert }d\mu _{n} \\
&\leq &\left( 1+N\right) \int_{B_{N}}\left\Vert \omega _{n}-\omega
\right\Vert _{\ell in}d\mu _{n}<\epsilon
\end{eqnarray*}%
When $n\rightarrow \infty ,$%
\begin{equation*}
\left\vert \lim_{n}\int_{TM}\omega _{n}d\mu _{n}-\int_{TM}\omega d\mu
\right\vert \leq \epsilon ,\forall \epsilon >0.
\end{equation*}%
Therefore%
\begin{equation*}
\lim_{n}\left\langle \omega _{n},\mu _{n}\right\rangle
=\lim_{n}\int_{TM}\omega _{n}d\mu _{n}=\int_{TM}\omega d\mu =\left\langle
\omega ,\mu \right\rangle .
\end{equation*}

\item $K_{N}$ is separated by $E.$ This follows from the duality and
approximation of continuous vector fields by smooth ones and the fact that $%
K_{N}$ is separated by $\Gamma ^{0}\left( M\right) ,$ that is: if $\mu ,\nu
\in K_{N},$ $\mu \neq \nu $, then there exists a $\omega _{0}\in $ $\Gamma
^{0}\left( M\right) $ such that $\mu \left( \omega _{0}\right) \neq \nu
\left( \omega _{0}\right) $ or
\begin{equation*}
\int_{TM}\omega _{0}d\mu \neq \int_{TM}\omega _{0}d\nu .
\end{equation*}%
%
%
%
%
\end{list}

The next ingredient regarding the steps followed by Bernard and Contreras is
the proof of injectivity of the map $\pi :\mathfrak{M}\left( L\right)
\rightarrow G$.

Recall that $\mathfrak{M}\left( L\right) $ the set of minimizing measures
for $L$ and $\widetilde{\mathcal{M}}_{0}\mathcal{=}\overline{\bigcup_{\mu
\in \mathfrak{M}\left( L\right) }\supp\mu }$ is the Mather set.

\begin{lemma}
Let $L$ be a Exact Magnetic Lagrangian. If $\mu $ and $\nu $ are two
distincts minimizing measures, then there exists a 1-form $\omega $ in $%
\Gamma ^{0}\left( M\right) $ such that%
\begin{equation*}
\int_{TM}\omega d\mu \neq \int_{TM}\omega d\nu
\end{equation*}
\end{lemma}

\begin{proof}
If $\mu \neq $ $\nu ,$ there exists $A$ in the Borel sigma algebra such that
$\mu \left( A\right) \neq $ $\nu \left( A\right) .$ We can suppose $A$ is a
closed set and $A\subset \supp\left( \mu \right) \cup \supp\left( \nu
\right) .$ The energy function for $L$ is given by $E\left( x,v\right) =%
\frac{1}{2}\left\Vert v\right\Vert ^{2}$ and since
\begin{equation*}
\supp\left( \mu \right) \cup \supp\left( \nu \right) \subset E^{-1}\left(
\alpha \left( 0\right) \right) =\left\{ \left( x,v\right) \in TM:\left\Vert
v\right\Vert ^{2}=2\alpha \left( 0\right) \right\} ,
\end{equation*}%
we have $A\subset E^{-1}\left( \alpha \left( 0\right) \right) .$ Moreover, $%
A\subset \widetilde{\mathcal{M}}_{0},$ where $\widetilde{\mathcal{M}}_{0}$
is the Mather set. By graph property $A$ is a graph on $\pi \left( A\right)
\ $and we can write%
\begin{equation*}
A=\left\{ \left( x,v\right) :x\in \pi \left( A\right) \text{ and }v=\pi
^{-1}\left( x\right) \right\}
\end{equation*}%
where $\pi ^{-1}$ is Lipschitz on the projected Mather set. Let%
\begin{equation*}
X\left( x\right) =\left\{
\begin{tabular}{l}
$\pi ^{-1}\left( x\right) ,$ if $x\in \pi \left( A\right) $ \\
$0,$ if $x\notin \pi \left( A\right) $%
\end{tabular}%
\right.
\end{equation*}%
and consider $f_{n}:M\rightarrow \left[ 0,1\right] $ sequence of smooth bump
functions%
\begin{equation*}
f_{n}\left( x\right) =\left\{
\begin{tabular}{l}
$1,$ if $x\in \pi \left( A\right) $ \\
$0,$ if $x\notin B_{n}\left( \pi \left( A\right) \right) $%
\end{tabular}%
\right.
\end{equation*}%
where $B_{n}\left( \pi \left( A\right) \right) $ is a neighborhood of the
compact $\pi \left( A\right) :$%
\begin{equation*}
B_{n}\left( \pi \left( A\right) \right) =\left\{ x\in M:d\left( x,a\right) <%
\frac{1}{n},\text{ for some }a\in \pi \left( A\right) \right\} .
\end{equation*}%
Let us consider $\overline{X}$ a continuous extension of $X|_{\pi \left(
A\right) }$ on $M.$ Then the vector field $X_{n}=f_{n}\overline{X}\in
\mathfrak{X}^{0}$$\left( M\right) ,$ converges pointwise to $X\left(
x\right) $ and
\begin{equation*}
\left\vert \left\langle X_{n}\left( x\right) ,v\right\rangle \right\vert
=\left\vert \left\langle f_{n}\overline{X}\left( x\right) ,v\right\rangle
\right\vert \leq \left\vert f_{n}\overline{X}\left( x\right) \right\vert
\left\vert v\right\vert \leq \left\vert \overline{X}\left( x\right)
\right\vert \left\vert v\right\vert .
\end{equation*}%
By Dominated Convergence Theorem
\begin{equation*}
\int_{TM}\left\langle X_{n}\left( x\right) ,v\right\rangle d\mu \rightarrow
\int_{TM}\left\langle X\left( x\right) ,v\right\rangle d\mu ,
\end{equation*}%
and%
\begin{equation*}
\int_{TM}\left\langle X_{n}\left( x\right) ,v\right\rangle d\nu \rightarrow
\int_{TM}\left\langle X\left( x\right) ,v\right\rangle d\nu .
\end{equation*}%
Suppose that for all $\omega \in \Gamma ^{0}\left( M\right) ,$
\begin{equation*}
\int_{TM}\omega d\mu =\int_{TM}\omega d\nu .
\end{equation*}%
Then we have%
\begin{equation*}
\int_{TM}\left\langle X_{n}\left( x\right) ,v\right\rangle d\mu
=\int_{TM}\left\langle X_{n}\left( x\right) ,v\right\rangle d\nu .
\end{equation*}%
Therefore%
\begin{eqnarray*}
\int_{TM}\left\langle X\left( x\right) ,v\right\rangle d\mu
&=&\int_{TM}\left\langle X\left( x\right) ,v\right\rangle d\nu \\
\int_{A}\left\langle X\left( x\right) ,v\right\rangle d\mu
&=&\int_{A}\left\langle X\left( x\right) ,v\right\rangle d\nu \\
\int_{A}\left\langle X\left( x\right) ,X\left( x\right) \right\rangle d\mu
&=&\int_{A}\left\langle X\left( x\right) ,X\left( x\right) \right\rangle d\nu
\\
\int_{A}2\alpha \left( 0\right) d\mu &=&\int_{A}2\alpha \left( 0\right) d\nu
\\
\alpha \left( 0\right) \mu \left( A\right) &=&\alpha \left( 0\right) \nu
\left( A\right) ,
\end{eqnarray*}%
Hence $\mu \left( A\right) =\nu \left( A\right) $ because $\alpha \left(
0\right) >0$ (See G. Paternain and M. Paternain in \cite{pat1})$.$ This
finishes the proof.
\end{proof}

The final step is entirely analogous to Lemma 9 of \cite{ber1} and we repeat
it here only for the sake of completeness. Mañé introduced a special type of
probability measures, the holonimic measures which is useful to prove
genericity results. A $C^{1}$ curve $\gamma :I\subset
\mathbb{R}
\rightarrow M$ of period $T>0$ define an element $\mu _{\gamma }\in F$ by%
\begin{equation*}
\mu _{\gamma }\left( f\right) =\frac{1}{T}\int_{0}^{T}f\left( \gamma \left(
s\right) ,\dot{\gamma}\left( s\right) \right) ds
\end{equation*}%
for each $f\in C.$ Let
\begin{equation*}
\Theta =\left\{ \mu _{\gamma }:\gamma \in C^{1}\left(
\mathbb{R}
,M\right) \text{ periodic of integral period}\right\} \subset F.
\end{equation*}%
The set $\mathcal{H}$ of holonomic probabilities is the closure of $\Theta $
in $F.$ One can see $\mathcal{H}$ is convex (see Mañé \cite{man2}). The
elements $\mu $ of $\mathcal{H}$ satisfy $\mu \left( 1\right) =1.$ We define
the compact $H_{N}\subset F$ as the set of holonomic probability measures
which are supported in $B_{N}.$ Therefore we have $\pi \left( H_{N}\right)
\subset K_{N}$.

The each Tonelli Lagrangian $L$ it is associated an element $\bar{L}\in
H_{N}^{\ast }$ as follows%
\begin{equation*}
\mu \mapsto \int_{TM}Ld\mu ,\mu \in H_{N}.
\end{equation*}%
Recalling that we have defined $M_{H_{N}}\left( L\right) $ as the set of
measures $\mu \in H_{N}$ which minimize the action $\int Ld\mu $ on $H_{N}.$

\begin{lemma}
If $L$ is a Exact Magnetic Lagrangian then there exists $N\in
\mathbb{N}
$ such that%
\begin{equation*}
\dim M_{K_{N}}\left( L\right) =\dim \mathfrak{M}\left( L\right) .
\end{equation*}
\end{lemma}

\begin{proof}
Mañé proves in \cite{man2} that $\mathfrak{M}\left( L\right) \subset
\mathcal{H}$. The Mather set $\widetilde{\mathcal{M}}_{0}$ is compact,
therefore $\mathfrak{M}\left( L\right) \mathcal{\subset }H_{N}$ for some $%
N\in
\mathbb{N}
.$ Mañé also proves in \cite{man2} that minimizing measures are also all the
minimizers of action functional $A_{L}\left( \mu \right) =\int Ld\mu $ on
the set of holonomic measures, therefore $\mathfrak{M}\left( L\right)
=M_{H_{N}}\left( L\right) $. By previous Lemma the map $\pi :\mathfrak{M}%
\left( L\right) \rightarrow G$ is injective, so that%
\begin{equation*}
\dim \pi \left( M_{H_{N}}\left( L\right) \right) =\dim \pi \left( \mathfrak{M%
}\left( L\right) \right) =\dim \mathfrak{M}\left( L\right)
\end{equation*}
\end{proof}

\begin{proof}
(of Theorem \ref{teo1}) Given $n\in \mathbb{N}$ apply Theorem \ref{teo2} and
obtain a residual subset $\mathcal{O}_{n}\left( A\right) \subset E=\Gamma
^{1}\left( M\right) $ such that%
\begin{equation*}
L\in A,\omega \in \mathcal{O}_{n}\left( A\right) \Rightarrow \dim
M_{Kn}\left( L+\omega \right) \leq \dim A.
\end{equation*}%
Let $\mathcal{O}\left( A\right) =\bigcap_{n}\mathcal{O}_{n}\left( A\right) .$
By the Baire property $\mathcal{O}\left( A\right) $ is residual. We have
that
\begin{equation*}
L\in A,\omega \in \mathcal{O}\left( A\right) ,n\in \mathbb{N}\Rightarrow
\dim M_{K_{n}}\left( L+\omega \right) \leq \dim A.
\end{equation*}%
Then by previous Lemma, $\dim \mathfrak{M}\left( L+\omega \right) \leq \dim
A $ for all $L\in A$ and all $\omega \in \mathcal{O}\left( A\right) .$ This
finishes the proof.\bigskip
\end{proof}

\section{ Some Dynamical Consequences\label{sec3}}

As it was pointed out in the Introduction, the Mather set $\widetilde{%
\mathcal{M}}_{c}$ associated to a cohomology class $c$ is contained in
another compact invariant set called the Aubry set $\widetilde{\mathcal{A}}%
_{c}$. It is also a graph over a compact subset of the manifold $M$ and it
is contained in the same energy level $\alpha (c)$ as $\widetilde{\mathcal{M}%
}_{c}$. Moreover, $\widetilde{\mathcal{A}}_{c}$ is chain recurrent set. All
these properties are proven in \cite{gon2}, see also \cite{fa1}.

In order to state the dynamical consequences of our Theorem \ref{teo1}, we
need to introduce the Aubry set and the concept of static classes for a
general Tonelli Lagrangian.

Let us consider the \textit{action} on a curve $\gamma :[0,T]\rightarrow M$
defined by

\begin{equation*}
\mathbb{A}_{L-c+k}\left( \gamma \right) =\int_{0}^{T}[L(\gamma ,\dot{\gamma}%
)-\eta (\gamma )\dot{(}\gamma )+k]dt
\end{equation*}%
where $k$ is a real number and $\eta $ is a representative of the class $c.$
The energy level $\alpha (c),$ namely Mañé's critical value of the
Lagrangian $L-c,$ may be characterized in several ways. $\alpha (c)$ is
defined by Mañé as the infimum of the numbers $k$ such that the action $%
\mathbb{A}_{L-c+k}\left( \gamma \right) $ is nonnegative for all closed
curve $\gamma :\left[ 0,T\right] \rightarrow M.$

Recall that, for a given real number $k$ the action potential $\Phi
_{L-c+k}:M\times M\rightarrow
\mathbb{R}
$ is defined by%
\begin{equation*}
\Phi _{L-c+k}\left( x,y\right) =\inf \mathbb{A}_{L-c+k}(\gamma )
\end{equation*}%
infimum taken over the curves $\gamma $ joining $x$ the $y$.

Mañé proved that%
\begin{equation*}
-\alpha (c)=\inf_{\mu \in \mathfrak{M}\left( L\right) }\int_{TM}\left(
L-\eta \right) d\mu ,
\end{equation*}%
where $\eta $ is a representative of the class $c$ and that $\alpha (c)$ is
the smallest number such that the action potential is finite, in other
words, if $k<\alpha (c)$, then $\Phi _{L-c+k}(x,y)=-\infty $ and for $k\geq
\alpha (c)$, $\Phi _{L-c+k}(x,y)\in {\mathbb{R}}$.

Observe that by Tonelli's Therorem (See for example in \cite{gon2}), for
fixed $t>0$, there always exists a minimizing extremal curve connecting $x$
to $y$ in time $t$. The potential calculates the global (or time
independent) infimum of the action. This value may not be realized by a
curve.

The potential $\Phi _{L-c+\alpha (c)}$ is not symmetric in general but%
\begin{equation*}
\delta _{M}\left( x,y\right) =\Phi _{L-c+\alpha (c)}\left( x,y\right) +\Phi
_{L-c+\alpha (c)}\left( y,x\right)
\end{equation*}%
is a pseudo-metric. A curve $\gamma :%
\mathbb{R}
\rightarrow M$ is called \textit{semistatic} if minimizes action between any
of its points:%
\begin{equation*}
\mathbb{A}_{L-c+\alpha (c)}\left( \gamma |_{\left[ a,b\right] }\right) =\Phi
_{L-c+\alpha (c)}\left( \gamma \left( a\right) ,\gamma \left( b\right)
\right) ,
\end{equation*}%
and $\gamma $ is called \textit{static} if is semistatic and $\delta
_{M}\left( \gamma \left( a\right) ,\gamma \left( b\right) \right) =0$ for
all $a,b\in
\mathbb{R}
.$

For example, the orbits contained in the Mather set $\widetilde{\mathcal{M}}%
_{c}$ project onto static curves. The Aubry set $\widetilde{\mathcal{A}}_{c}$
is the set of the points $\left( x,v\right) \in TM$ such that the projection
$\gamma \left( t\right) =\pi \circ \varphi _{t}\left( x,v\right) $ is a
static curve, where $\varphi _{t}$ is the Euler-Lagrange flow. We just saw
that the Mather set ${\widetilde{\mathcal{M}}}_{c}$ is contained in the
Aubry set $\widetilde{\mathcal{A}}_{c}.$

Denoting the projected Aubry set by $\mathcal{A}_{c}$, the function $\delta
_{M}|_{\mathcal{A}_{c}\mathcal{\times A}_{c}}:\mathcal{A}_{c}\mathcal{\times
\mathcal{A}}_{c}\mathcal{\rightarrow
\mathbb{R}
}$ is called \textit{Mather semi-distance.} We define the \textit{quotient
Aubry set} $\left( \mathcal{A}_{M},\delta _{M}\right) $ to be the metric
space by identifying two points $x,y\in \mathcal{A}_{c}$ if their
semi-distance $\delta _{M}\left( x,y\right) $ vanishes. When we consider $%
\delta _{M}$ on the quotient space $\mathcal{A}_{M}$ we will call it the
\textit{Mather distance} and the elements of $\mathcal{A}_{M}$ are called
\textit{static classes }for $L-c$. Observe that the static classes are
disjoint subsets of the energy level set $\alpha (c)$ and a static curve is
in the same static class.

Then we have the following corollary of the Theorem \ref{teo1}:

\begin{corollary}
Let $L$ be a Exact Magnetic Lagrangian. Then there exists a residual subset $%
\mathcal{O}$ of $\Gamma ^{1}\left( M\right) $ such that for all $c\in
H^{1}\left( M,%
\mathbb{R}
\right) \ $and for all $\omega \in \mathcal{O}$, the Lagrangian $L+\omega -c$
has at most $1+\dim H^{1}\left( M,%
\mathbb{R}
\right) $ static classes.
\end{corollary}

\begin{proof}
It suffices to show that each static class supports at least one ergodic
minimizing measure. In fact, let $\Lambda $ be a static class for $L+\omega
-c$ and $\left( p,v\right) \in \widetilde{\mathcal{A}}_{c}$ with $p\in
\Lambda .$ For $T>0$ we define a Borel probability measure $\mu _{T}$ on $TM$
by%
\begin{equation*}
\mu _{T}\left( f\right) =\frac{1}{T}\int_{0}^{T}f\left( \varphi _{s}\left(
p,v\right) \right) ds
\end{equation*}%
All these probability measures have their supports contained in $\widetilde{%
\mathcal{A}}_{c}$ that is a compact subset, consequently, we can extract a
sequence $\mu _{T_{n}}$ weakly convergent to $\mu $:%
\begin{equation*}
\mu \left( f\right) =\lim_{T\rightarrow \infty }\frac{1}{T_{n}}%
\int_{0}^{T_{n}}f\left( \varphi _{s}\left( p,v\right) \right) ds,
\end{equation*}%
which is a ergodic minimizing measure whose support is contained in $\Lambda
$ (See \cite{fa1} for details).
\end{proof}

Now we present some dynamical consequences assuming that the Lagrangian $L$
has finitely many static classes. In this manner, by previous corollary, the
properties presented here are generic on set of Exact Magnetic for all
cohomology class.

The projected Aubry set $\mathcal{A}_{c}$ is chain recurrent and the static
classes are connected so they are the connected components of $\mathcal{A}%
_{c}$. Moreover the static classes are the chain transitive components of $%
\mathcal{A}_{c}$ and we obtain the following cycle property: If two supports
of ergodic minimizing measures are contained in a static class, then there
exists a cycle consisting of static curves in the same static class
connecting them.

Contreras and Paternain prove in \cite{gon1} that between two static classes
there exists a chain of static classes connected by heteroclinic semistatic
orbits. More precisely they show

\begin{theorem}
Suppose that the number of static classes is finite. Then given two static
classes $\Lambda _{k}$ and $\Lambda _{l},$ there exist classes $\Lambda
_{1}=\Lambda _{k},\Lambda _{2},...,\Lambda _{n}=\Lambda _{l}$ and $\theta
_{1},\theta _{2},...,\theta _{n-1}\in TM$ such that for all $i=1,...,n-1$ we
have that $\gamma _{i}\left( t\right) =\pi \circ \varphi _{t}\left( \theta
_{i}\right) $ are semistatic curves, $\alpha \left( \theta _{i}\right)
\subset \Lambda _{i}$ and $\omega \left( \theta _{i}\right) \subset \Lambda
_{i+1}.$
\end{theorem}

Another important property, demonstrated by P. Bernard in \cite{ber2}, is
the semi-continuity of the Aubry set%
\begin{equation*}
H^{1}\left( M,%
\mathbb{R}
\right) \ni c\mapsto \widetilde{\mathcal{A}}_{c},
\end{equation*}%
when $\mathcal{A}_{M}$ is finite. In order to be more precise he showed the
following Theorem

\begin{theorem}
Let $L_{k}$ be a sequence of Tonelli Lagrangians converging to $L.$ Then
given a neighborhood $U$ of $\widetilde{\mathcal{A}}_{0}$ in $TM,$ there
exists $k_{0}$ such that $\widetilde{\mathcal{A}}_{0}\left( L_{k}\right)
\subset U$ for each $k\geq k_{0},$ where $\widetilde{\mathcal{A}}_{0}\left(
L_{k}\right) $ is the Aubry set for the Lagrangian $L_{k}.$
\end{theorem}

In fact Bernard showed that this Theorem is true with a weaker hypothesis
than $\mathcal{A}_{M}$ be finite, namely coincidence hypothesis (See \cite%
{ber2}).

\section{Example}

In this section we present an example of a Exact Magnetic Lagrangian on flat
torus $\mathbb{T}^{2}$ whose quotient Aubry set $\mathcal{A}_{M}$ is a
Cantor set, therefore not every Exact Magnetic Lagrangian has finitely many
static classes.

Let $L:T\mathbb{T}^{2}\rightarrow
\mathbb{R}
$ be a Exact Magnetic Lagrangian defined by%
\begin{equation*}
L\left( x,y,v_{1},v_{2}\right) =\frac{\left\Vert \left( v_{1},v_{2}\right)
\right\Vert ^{2}}{2}+\left\langle \left( 0,f\left( x\right) \right) ,\left(
v_{1},v_{2}\right) \right\rangle ,
\end{equation*}%
where $f$ is a $C^{2}$ nonpositive and periodic function whose set of
minimum points $\Gamma _{\min }$ is a Cantor set and $f|_{\Gamma _{\min }}$
is a negative constant.

%
%

In this case the system of Euler-Lagrange is given by

\begin{equation*}
\left\{
\begin{array}{l}
\dot{x}=v \\
\dot{v}=-f^{\prime }\left( x\right) Jv%
\end{array}%
\right.
\end{equation*}%
where $J$ is the canonical sympletic matrix.

\begin{lemma}
The Mañé's critical values of $L$ is $\alpha (0)=f\left( a\right) ^{2}/2,$
where $a\in \Gamma _{\min }.$ Moreover, the closed curves $\gamma _{a}$
defined by $\gamma _{a}\left( t\right) =\left( a,-f\left( a\right) t\right)
, $ are static curves.
\end{lemma}

\begin{proof}
Given any curve $\beta \left( t\right) =\left( x\left( t\right) ,y\left(
t\right) \right) $ on $\mathbb{T}^{2},$ we have%
\begin{equation}
L\left( \beta ,\dot{\beta}\right) =\frac{\dot{x}^{2}+\dot{y}^{2}}{2}+f\left(
x\right) \dot{y}\left( t\right) =\frac{\left( \dot{y}+f\left( x\right)
\right) ^{2}+\dot{x}^{2}}{2}-\frac{f\left( x\right) ^{2}}{2}\geq -\frac{%
f\left( a\right) ^{2}}{2}.  \label{eq6}
\end{equation}%
Then%
\begin{equation*}
\mathbb{A}_{L+f\left( a\right) ^{2}/2}\left( \beta \right)
=\int_{0}^{T}\left( L\left( \beta ,\dot{\beta}\right) +\frac{f\left(
a\right) ^{2}}{2}\right) dt\geq 0,
\end{equation*}%
and we obtain $\alpha (0)\leq \frac{f\left( a\right) ^{2}}{2}.$ Observe that
if $0<k\leq \frac{f\left( a\right) ^{2}}{2},$ the closed curve given by $%
\gamma _{k}\left( t\right) =\left( a,\sqrt{2k}t\right) ,$ where $a\in \Gamma
_{\min },$ is Euler-Lagrange solution and its energy is $E=k$. Moreover,%
\begin{equation*}
\mathbb{A}_{L+k}\left( \gamma _{k}\right) =\int_{0}^{T}\left( L\left( \gamma
_{k},\dot{\gamma}_{k}\right) +k\right) dt=\int_{0}^{T}\left( 2k+f\left(
a\right) \sqrt{2k}\right) dt.
\end{equation*}%
Therefore%
\begin{equation*}
\mathbb{A}_{L+k}\left( \gamma _{k}\right) <0\text{ if }k<\frac{f\left(
a\right) ^{2}}{2}\text{ and }\mathbb{A}_{L+k}\left( \gamma _{k}\right) =0%
\text{ if }k=\frac{f\left( a\right) ^{2}}{2}.
\end{equation*}%
This shows that $\alpha (0)=\frac{f\left( a\right) ^{2}}{2}$ and the curve $%
\gamma _{a}$ is semistatic, i.e., realizes the action potential. Since $%
\gamma _{a}$ is a semistatic closed curve, it is static curve.
\end{proof}

To complete the example, it suffices to show that the application $\Psi
:\Gamma _{\min }\rightarrow \mathcal{A}_{M},$ given by $\Psi \left( a\right)
=\left[ \left( a,0\right) \right] ,$ where $\left[ \left( a,0\right) \right]
$ is a representative of the static class conteining the curve $\gamma _{a},$
is a Lipschitz bijection. In fact, since the action potential $\Phi
_{L+\alpha (0)}$ is Lipschitz, the distance $\delta _{M}$ on quotient Aubry
set $\mathcal{A}_{M}$ also is Lipschitz.

In order to show the surjectivity of $\Psi $ it is enough to show that the
projected Aubry set $\mathcal{A}_{0}$ is exactly the union of the closed
curves $\gamma _{a}$ with $a\in \Gamma _{\min }.$ Suppose that there exists $%
p\in \mathcal{A}_{0}$ such that $\pi \left( p\right) \notin \Gamma _{\min },$
where $\pi $ is the canonical projection of $\mathbb{T}^{2}$ on $%
\mathbb{R}
/%
\mathbb{Z}
.$ Then there exists a neighborhood $V_{p}$ of $p$ such that $f\left(
x\right) >f\left( a\right) $ for all $a\in \Gamma _{\min }$ and $x\in V_{p}$%
. Let $\gamma $ be a piece, contained in $V_{p},$ of the static curve
passing through $p$. The inequality \ref{eq6} implies $\mathbb{A}_{L+\alpha
(0)}\left( \gamma \right) >0.$ Moreover, it follows by inequality \ref{eq6}
which the action $L+\alpha (0)$ of any curve is nonnegative, so $\Phi
_{L+\alpha (0)}\left( x,y\right) \geq 0$ for all $x,y\in \mathbb{T}^{2}.$
Then%
\begin{equation*}
\mathbb{A}_{L+\alpha (0)}\left( \gamma \right) =\Phi _{L+\alpha (0)}\left(
\gamma \left( 0\right) ,\gamma \left( T\right) \right) =-\Phi _{L+\alpha
(0)}\left( \gamma \left( 0\right) ,\gamma \left( T\right) \right) \leq 0.
\end{equation*}%
This is a contradiction.

If $\Psi $ is not injective there exists $b\in \Gamma _{\min },b\neq a$ such
that $\left( b,0\right) \in \left[ \left( a,0\right) \right] .$ Since each
static class is connected (See G. Contreras and G. Paternain in \cite{gon1},
Proposition 3.4) and $b\in \pi \left( \left[ \left( a,0\right) \right]
\right) $ we have that $\pi \left( \left[ \left( a,0\right) \right] \right)
\subset
\mathbb{R}
/%
\mathbb{Z}
$ is connected so it is an interval. By total disconnectedness of $\Gamma
_{\min }$, there exists $q\in \pi \left( \left[ \left( a,0\right) \right]
\right) -\Gamma _{\min }.$ The contradiction follows of the inequality \ref%
{eq6} by same argument above.

\bigskip\

\end{document}